\documentclass{article}%
\usepackage{graphicx}
\usepackage{amsmath}
\usepackage{amsfonts}
\usepackage{amssymb}%
\providecommand{\U}[1]{\protect\rule{.1in}{.1in}}
\providecommand{\U}[1]{\protect\rule{.1in}{.1in}}
\providecommand{\U}[1]{\protect\rule{.1in}{.1in}}
\providecommand{\U}[1]{\protect\rule{.1in}{.1in}}
\providecommand{\U}[1]{\protect\rule{.1in}{.1in}}
\providecommand{\U}[1]{\protect\rule{.1in}{.1in}}
\providecommand{\U}[1]{\protect\rule{.1in}{.1in}}
\providecommand{\U}[1]{\protect\rule{.1in}{.1in}}
\providecommand{\U}[1]{\protect\rule{.1in}{.1in}}
\providecommand{\U}[1]{\protect\rule{.1in}{.1in}}
\providecommand{\U}[1]{\protect\rule{.1in}{.1in}}
\providecommand{\U}[1]{\protect\rule{.1in}{.1in}}
\providecommand{\U}[1]{\protect\rule{.1in}{.1in}}
\providecommand{\U}[1]{\protect\rule{.1in}{.1in}}
\providecommand{\U}[1]{\protect\rule{.1in}{.1in}}
\providecommand{\U}[1]{\protect\rule{.1in}{.1in}}
\providecommand{\U}[1]{\protect\rule{.1in}{.1in}}
\providecommand{\U}[1]{\protect\rule{.1in}{.1in}}
\providecommand{\U}[1]{\protect\rule{.1in}{.1in}}
\providecommand{\U}[1]{\protect\rule{.1in}{.1in}}
\providecommand{\U}[1]{\protect\rule{.1in}{.1in}}
\providecommand{\U}[1]{\protect\rule{.1in}{.1in}}
\providecommand{\U}[1]{\protect\rule{.1in}{.1in}}
\providecommand{\U}[1]{\protect\rule{.1in}{.1in}}
\providecommand{\U}[1]{\protect\rule{.1in}{.1in}}
\providecommand{\U}[1]{\protect\rule{.1in}{.1in}}
\providecommand{\U}[1]{\protect\rule{.1in}{.1in}}
\providecommand{\U}[1]{\protect\rule{.1in}{.1in}}
\providecommand{\U}[1]{\protect\rule{.1in}{.1in}}

\newtheorem{theorem}{Theorem}
{}

\newtheorem{conclusion}{Conclusion}

\newtheorem{corollary}{Corollary}

\newtheorem{definition}{Definition}
\newtheorem{example}{Example}

{}

\newtheorem{proposition}{Proposition}
\newtheorem{remark}{Remark}

\newtheorem{summary}{Summary}
\newenvironment{proof}[1][Proof]{\textbf{#1.} }{\ \rule{0.5em}{0.5em}}

\begin{document}

\title{On the spectral properties of the Schrodinger operator with a periodic
PT-symmetric potential}
\author{O. A. Veliev\\{\small \ Depart. of Math., Dogus University, }\\{\small Ac\i badem, 34722, Kadik\"{o}y, \ Istanbul, Turkey.}\\\ {\small e-mail: oveliev@dogus.edu.tr}}
\date{}
\maketitle

\begin{abstract}
In this paper we investigate the spectrum and spectrality of the
one-dimensional Schrodinger operator with a periodic PT-symmetric
complex-valued potential.



\end{abstract}

\section{ Introduction and Preliminary Facts}

In this paper we investigate the one dimensional Schr\"{o}dinger operator
$L(q)$ generated in $L_{2}(-\infty,\infty)$ by the differential expression
\begin{equation}
-y^{^{\prime\prime}}(x)+q(x)y(x),
\end{equation}
where $q$ is complex-valued, locally integrable, periodic and PT-symmetric.
Without loss of generality, we assume that the period of $q$ is $1$ and the
integral of $q$ over $[0,1]$ is zero. Thus%
\begin{equation}
q\in L_{1}[0,1],\text{ }\int_{0}^{1}q(x)dx=0,\text{ }q(x+1)=q(x),\text{
}\overline{q(-x)}=q(x)\text{ (a.e.).}%
\end{equation}

A basic mathematical question of PT-symmetric quantum mechanics concerns the
reality of the spectrum of the considered Hamiltonian (see [2, 15 and
references of them]). In the first papers [1,3,5,6,10] about the PT-symmetric
periodic potential, the appearance and disappearance of real energy bands for
some complex-valued PT-symmetric periodic potentials under perturbations have
been reported. Shin [17] showed that the appearance and disappearance of such
real energy bands imply the existence of nonreal band spectra. He involved
some condition on the Hill discriminant to show the existence of nonreal
curves in the spectrum. Caliceti and Graffi [4] found explicit condition on
the Fourier coefficient of the potential providing the nonreal spectra for
small potentials. Besides, they proved that if all gaps of the spectrum of the
Hill operator $L(q)$ with distributional potential $q$ are open and the width
of the $n$-th gap does not vanish as $n\rightarrow\infty,$ then the spectrum
of $L(q)+gW,$ where $W$ is a bounded periodic function and $g$ is a small
number, is real. This result can not be used for the locally integrable
periodic potentials (2), since the width of the $n$-th gap vanishes as
$n\rightarrow\infty$.

In this paper, we first consider the general spectral property of the spectrum
of $L(q)$ under conditions (2) and prove that the main part of its spectrum is
real and contains the large part of $[0,\infty).$ Using this we find necessary
and sufficient condition on the potential for finiteness of the number of the
nonreal arcs in the spectrum of $L(q)$. Besides we find necessary and
sufficient conditions for the equality of the spectrum of $L(q)$ to the half
line. Moreover, we consider the connections between spectrality of $L(q)$ and
the reality of its spectrum for some class of PT-symmetric periodic
potentials. Finally, we find explicit conditions on the potential $q$ for
which the number of gaps in the real part of the spectrum of $L(q)$ is finite.

Now let us list the well-known results, as Summary 1-Sammary 7, about $L(q)$
which will be essentially used in this paper. Note that we formulate the
well-known results from the books [7, 19] and the papers [11, 13, 14, 16,17]
in the suitable form for this paper and by using the unique notation, since
the different notations were used in those references.

\begin{summary}
The spectrum $\sigma(L)$ of the operator $L$ is the union of the spectra
$\sigma(L_{t})$ of the operators $L_{t}$ for $t\in(-\pi,\pi]$ generated in
$L_{2}[0,1]$ by (1) and the boundary conditions
\begin{equation}
y(1)=e^{it}y(0),\text{ }y^{^{\prime}}(1)=e^{it}y^{^{\prime}}(0).
\end{equation}

\end{summary}

\begin{summary}
The eigenvalues of $L_{t}$ are the roots of the characteristic equation
\begin{equation}
F(\lambda)=2\cos t,
\end{equation}
where $F(\lambda):=\varphi^{\prime}(1,\lambda)+\theta(1,\lambda)$ is the Hill
discriminant, $\theta$ and $\varphi$ are the solutions of
\begin{equation}
-y^{^{\prime\prime}}(x)+q(x)y(x)=\lambda y(x)
\end{equation}
satisfying the initial conditions $\theta(0,\lambda)=\varphi^{\prime
}(0,\lambda)=1,$ $\theta^{\prime}(0,\lambda)=\varphi(0,\lambda)=0.$
\end{summary}

The following 2 summaries immediately follows from summaries 1 and 2.

\begin{summary}
$\sigma(L_{-t})=\sigma(L_{t})=\sigma(L_{t+2\pi})$ and $\lambda\in\sigma
(L_{t})$ if and only if there exist a solution $\Psi(x,\lambda)$ of (5)
satisfying $\Psi(x+1,\lambda)=e^{it}\Psi(x,\lambda)$ for all $x.$
\end{summary}

\begin{summary}
\textbf{ }$\lambda\in\sigma(L)$ if and only if $F(\lambda)\in\lbrack$\ $-2,2]$.
\end{summary}

The following summary was proved in [16, 20]

\begin{summary}
The spectrum $\sigma(L)$ consist of analytic arcs whose endpoints are the
eigenvalues of $L_{t}$ for $t=0,\pi$ and the multiple eigenvalues of $L_{t}$
for $t\in(0,\pi).$ Moreover $\sigma(L)$ does not contain the closed curves,
that is, the resolvent set $\mathbb{C}\backslash\sigma(L)$ is connected.
\end{summary}

Let (2) holds. Then using Summary 3 one can readily see that the function
$\Phi(x,\lambda):=$ $\overline{\Psi(-x,\lambda)}$ satisfies the equation
$-y^{^{\prime\prime}}(x)+q(x)y(x)=\overline{\lambda}y(x)$ and the equality
\[
\Phi(x+1,\lambda)=\overline{\Psi(-x-1),\lambda)}=\overline{e^{-it}%
\Psi(-x,\lambda)}=e^{it}\Phi(x,\lambda).
\]
It means that $\overline{\lambda}\in\sigma(L_{t}).$ Hence\textbf{ }we have

\begin{summary}
If (2)holds then the following implications hold%
\[
\lambda\in\sigma(L_{t})\Longrightarrow\overline{\lambda}\in\sigma
(L_{t}),\text{ }\lambda\in\sigma(L)\Longrightarrow\overline{\lambda}\in
\sigma(L).
\]
The first and second implications were proved in [11] and [17] respectively.
\end{summary}

The following summary is proved in [17] (see Theorem 3 and Corollary 4).

\begin{summary}
Suppose (2) holds. Then $F(\lambda)$ is real for all real $\lambda,$ where
$F(\lambda)$ is defined in (4). If $F(\lambda)\in(-2,2)$ and $F^{^{\prime}%
}(\lambda)=0$ or $F(\lambda)=\pm2$, $F^{^{\prime}}(\lambda)=0$ and
$F^{^{\prime\prime}}(\lambda)=0$ then $\sigma(L(q))$ contains nonreal number
in any neighborhood of $\lambda$.
\end{summary}

In Section 2 we use these summaries and the uniform asymptotic formulas for
the eigenvalues of $L_{t}(q)$ obtained in [22] and [23] and formulated as
Summary 8 and Summary 9 in order to describe the spectrum of the operator
$L(q)$ with the potential (2). In the last section using the results of
Section 2 and the papers [23, 27] we consider in detail the shape of the
spectrum and the asymptotic connection of reality of $\sigma(L)$ and
spectrality of $L$.

\section{On the Spectrum of $L(q)$}

\textbf{ }In the case $q=0$ the eigenvalues and eigenfunctions of $L_{t}(q)$
are $(2\pi n+t)^{2}$ and $e^{i(2\pi n+t)x}$ for $n\in\mathbb{Z}$ respectively.
In [27] we proved that the eigenvalues of $L_{t}$ can be numbered (counting
the multiplicity) by elements of $\mathbb{Z}$ such that, for each $n$ the
function $\lambda_{n}(t)$ is a piecewise differentiable on $[0,\pi]$ and
$\lambda_{n}(-t)=\lambda_{n}(t)$ (see Summary 3). Thus
\begin{equation}
\sigma(L(q))=%
{\textstyle\bigcup\limits_{n\in\mathbb{Z}}}
\Gamma_{n},
\end{equation}
where $\Gamma_{n}=\left\{  \lambda_{n}(t):t\in\lbrack0,\pi]\right\}  $ (see
(16) of [27]). In future we call $\Gamma_{n}$ as band of $\sigma
(L(q))$.\ Moreover, it follows from the construction of $\Gamma_{n}$ that
\begin{equation}
\Gamma_{n}\cap\sigma\left(  L_{t}(q)\right)  =\left\{  \lambda_{n}(t)\right\}
,\text{ }\left\vert \Gamma_{n}\right\vert <\infty,\text{ }\lim_{n\rightarrow
\infty}d(\Gamma_{n},0)=\infty,
\end{equation}
where $\left\vert \Gamma_{n}\right\vert $ is the length of $\Gamma_{n}$ and
$d(\Gamma_{n},0)$ is the distance from $\Gamma_{n}$ to origin. Besides one can
readily see the following properties of the bands $\Gamma_{n}$.

\begin{proposition}
$(a)$ $\Gamma_{n}$ is a single open curve with the end points $\lambda_{n}(0)$
and $\lambda_{n}(\pi)$.

$(b)$ Two bands $\Gamma_{n}$ and $\Gamma_{m}$ may have at most one common point.
\end{proposition}

\begin{proof}
$(a)$ It follows from (3) and (4) that $\lambda_{n}(t_{1})\neq\lambda
_{n}(t_{2})$ for $0\leq t_{1}<t_{2}\leq\pi.$ Therefore $\Gamma_{n}$ is a
single open curve and hence has two end points. Since $\lambda_{n}%
:[0,\pi]\rightarrow\Gamma_{n}$ is a piecewise differentiable functions it maps
the ends of $[0,\pi]$ to the ends of $\Gamma_{n}.$

$(b)$ The inequality $\lambda_{n}(t_{1})\neq\lambda_{m}(t_{2})$ for $t_{1}\neq
t_{2}$ implies that the bands $\Gamma_{n}$ and $\Gamma_{m}$ has a common point
if and only if there exist $t\in\lbrack0,\pi]$ such that $\lambda
_{n}(t)=\lambda_{m}(t).$ It means that $\lambda_{n}(t)$ is a multiple
eigenvalue of $L_{t}$ and hence is the roots of $F^{^{\prime}}(\lambda)=0$,
where $F(\lambda)$ is the Hill discriminant defined in Summary 2. Since
$F(\lambda)$ is a nonzero entire function, the set of zeros of $F^{^{\prime}%
}(\lambda)$ is a discrete set and has no finite limit points. Therefore, by
the inequality in (7), two bands can have only finite number common points. If
the common points of $\Gamma_{n}$ and $\Gamma_{m}$ are more than one, then
there exist $t_{1}$ and $t_{2}$ such that $0\leq t_{1}<t_{2}\leq\pi$ and
\[
\lambda_{n}(t_{1})=\lambda_{m}(t_{1}),\text{ }\lambda_{n}(t_{2})=\lambda
_{m}(t_{2}),\text{ }\lambda_{n}(t_{1})\neq\lambda_{n}(t_{2}),\text{ }%
\lambda_{n}(t)\neq\lambda_{m}(\tau)
\]
for all $t\in(t_{1},t_{2})$ and $\tau\in(t_{1},t_{2}).$ However, it implies
that
\[
\left\{  \lambda_{n}(t):t\in\lbrack t_{1},t_{2}]\right\}  \cup\left\{
\lambda_{m}(t):t\in\lbrack t_{1},t_{2}]\right\}
\]
is a closed curve which contradicts Summary 5.
\end{proof}

\begin{remark}
Note that if $\lambda_{n}(t)$ for $t\in(0,\pi)$ and $t=0,\pi$ is a multiple
eigenvalue of $L_{t}(q)$ of multiplicity $p,$ then $p$ bands of the spectrum
have common interior and end point $\lambda_{n}(t)$ respectively. In
particular, $\lambda_{n}(t)$ is a simple eigenvalue if and only if it belong
only to one band $\Gamma_{n}$.
\end{remark}

These arguments with summaries 5 and 6 yield the following

\begin{theorem}
Suppose that (2) holds.

$(a)$ If $\lambda_{n}(t_{1})$ and $\lambda_{n}(t_{2})$ are real numbers, where
$0\leq t_{1}<t_{2}\leq\pi$ then $\gamma:=\left\{  \lambda_{n}(t):t\in\lbrack
t_{1},t_{2}]\right\}  $ is an interval of the real line.

$(b)$ If the eigenvalues $\lambda_{n}(0)$ and $\lambda_{n}(\pi)$ are real
numbers then $\lambda_{n}(t)$ are real eigenvalues of $L_{t}(q)$ for all
$t\in(0,\pi)$, that is, $\Gamma_{n}$ is either $[\lambda_{n}(0),\lambda
_{n}(\pi)]$ or $[\lambda_{n}(\pi),\lambda_{n}(0)]$.

$(c)$ The spectrum of $L(q)$ is completely real if and only if all eigenvalues
of $L_{0}(q)$ and $L_{\pi}(q)$ are real, that is, the roots of $\left(
F(\lambda)\right)  ^{2}=4$ are the real numbers.
\end{theorem}

\begin{proof}
$(a)$ If $\gamma$ is not an interval of the real line, then, by Summary 6,
$\sigma(L)$ contains a closed curve $\gamma\cup\widetilde{\gamma},$ where
$\widetilde{\gamma}=\left\{  \lambda\in\mathbb{C}:\overline{\lambda}\in
\gamma\right\}  ,$ which contradicts Summary 5.

$(b)$ Since $\lambda_{n}(0)$ and $\lambda_{n}(\pi)$ are the end points of
$\Gamma_{n}$, $(b)$ follows from $(a).$

$(c)$ If the spectrum of $L(q)$ is real, then by Summary 1 the eigenvalues of
$L_{0}(q)$ and $L_{\pi}(q)$ are real. Now suppose that all eigenvalues of
$L_{0}(q)$ and $L_{\pi}(q)$ are real. It means that the eigenvalues
$\lambda_{n}(0)$ and $\lambda_{n}(\pi)$ are real numbers for all $n.$ Then, by
$(b),$ $\Gamma_{n}$ for all $n$ and hence, by (6), $\sigma(L)$ is completely real.
\end{proof}

Now to consider, the reality of the spectrum in detail, we investigate the
points in which the spectrum ceases to be real. These points are crucial and
can be defined as follows.

\begin{definition}
A real number $\lambda\in\sigma(L)$ is said to be a left (right) complexation
point of the spectrum if there exists $\varepsilon>0$ such that $[\lambda
,\lambda+\varepsilon]\subset$ $\sigma(L)$ ( $[\lambda-\varepsilon
,\lambda]\subset$ $\sigma(L))$ and $\sigma(L)$ contains a nonreal number in
any neighborhood of $\lambda.$ Both left and right complexation points are
called complexation points.
\end{definition}

\begin{theorem}
Suppose that (2) holds. If $\lambda_{n}(t)$ is a complexation point, then it
is a multiple eigenvalue. Moreover, the multiplicity of $\lambda_{n}(t)$ is
greater than $\ 2$ if $t=0,\pi.$
\end{theorem}

\begin{proof}
Let $\lambda_{n}(t)$ be a complexation point. If $\lambda_{n}(t)$ is a simple
eigenvalue, then by Remark 1, $\lambda_{n}(t)$ belong to $\Gamma_{n}$ and does
not belong to $\Gamma_{m}$ for $m\neq n$. Moreover, there exists a
neighborhood $U:=\left\{  \lambda\in\mathbb{C}:\left\vert \lambda-\lambda
_{n}(t)\right\vert <\delta\right\}  $ of $\lambda_{n}(t)$ such that
$U\cap\Gamma_{m}=\varnothing$ for all $m\neq n.$ Then by Definition 1 and
Summary 6 there exist a nonreal number $\lambda\in U$ and $t_{0}\in(0,\pi)$
such that $\lambda\in\left(  \Gamma_{n}\cap\sigma\left(  L_{t_{0}}(q)\right)
\right)  $ and $\overline{\lambda}\in\left(  \Gamma_{n}\cap\sigma\left(
L_{t_{0}}(q)\right)  \right)  $ which contradicts (7).

Now suppose that $\lambda_{n}(0)$ is a complexation point and multiple
eigenvalue of multiplicity two. Then, by Remark 1, there exists $m\neq n$ such
that $\lambda_{m}(0)=\lambda_{n}(0)$ and $\lambda_{k}(0)\neq\lambda_{n}(0)$
for all $k\neq n,m.$ Therefore using Definition 1 and taking into account the
continuity of $\lambda_{k}(t)$ for all $k\in\mathbb{Z},$ we see that there
exist $\varepsilon>0$ and $\delta>0$ such that at least one of the sets
$\left\{  \lambda_{n}(t):t\in\lbrack0,\varepsilon]\right\}  $ and $\left\{
\lambda_{m}(t):t\in\lbrack0,\varepsilon]\right\}  $ is the interval of the
real line and
\begin{equation}
\left\{  \lambda\in\mathbb{C}:\left\vert \lambda-\lambda_{n}(0)\right\vert
<\delta\right\}  \cap\left\{  \lambda_{k}(t):t\in\lbrack0,\varepsilon
]\right\}  =\varnothing,\text{ }\forall k\neq n,m.
\end{equation}
Suppose, without loss of generality, that $\left\{  \lambda_{n}(t):t\in
\lbrack0,\varepsilon]\right\}  \subset\mathbb{R}$. Then by (8) and Definition
1 there exist $t\in\lbrack0,\varepsilon]$ and $\delta>0$ such that
$\lambda_{m}(t)\in\left\{  \lambda\in\mathbb{C}:\left\vert \lambda-\lambda
_{n}(0)\right\vert <\delta\right\}  \backslash\mathbb{R}$. These arguments
with (7), (8) and Summary 6 imply that $\overline{\lambda_{m}(t)}$ belong to
$\sigma(L(q))$ and does not belong to $\Gamma_{k}$ for any $k\in\mathbb{Z}$.
It contradicts (6). In the same way we consider $\lambda_{n}(\pi)$
\end{proof}

Now we consider the large eigenvalues by using the uniform asymptotic formulas
for the Bloch eigenvalues obtained in [22, 23] and then combining it with the
above results we will describe the spectrum lying outside of some disk. In
[22] (see Theorem 2) we proved that the large eigenvalues of the operators
$L_{t}(q)$\ for $t\neq0,\pi$ consist of the sequence $\left\{  \lambda
_{n}(t):\mid n\mid\gg1\right\}  $ satisfying\ \
\begin{equation}
\lambda_{n}(t)=(2\pi n+t)^{2}+O(n^{-1}\ln\left\vert n\right\vert )
\end{equation}
as $n\rightarrow\infty$ and the formula (9) is uniform with respect to $t$ in
$[h,\pi-h],$ where $h$ is a fixed number and, without loss of generality, it
is assumed that the integral of $q$ over $[0,1]$ is $0.$ In future we assume
that $h\in(0,1/15\pi).$ The following summary follows from (9).

\begin{summary}
For any fixed $h$, there exists an integer $N(h)$ and positive constant $M(h)$
such that for $|n|>N(h)$ and $t\in\lbrack h,\pi-h]$ there exists unique,
counting multiplicity, eigenvalue $\lambda_{n}(t)$ satisfying
\begin{equation}
\left\vert \lambda_{n}(t)-(2\pi n+t)^{2}\right\vert \leq M(h)n^{-1}\ln
n,\text{ }\left\vert \lambda_{n}(t)-\lambda_{k}(t)\right\vert >2\pi
^{2}n,\text{ }\forall k\neq n.
\end{equation}
The eigenvalue $\lambda_{n}(t)$ is simple for all $|n|>N(h)$ and $t\in\lbrack
h,\pi-h].$
\end{summary}

Moreover, as it was shown in [23] (see (20) and Remark 2.11of [23]) and [27]
(see page 60 of [27]) $N(h)$ can be chosen so that the following holds.

\begin{summary}
For $t\in\lbrack0,h]$ and $n>N(h)$ there exist two eigenvalues, counting
multiplicity, denoted by $\lambda_{n}(t)$ and $\lambda_{-n}(t)$ and
satisfying
\begin{equation}
\left\vert \lambda_{\pm n}(t)-(2\pi n+t)^{2}\right\vert \leq15\pi nh<n,\text{
}\left\vert \lambda_{\pm n}(t)-\lambda_{k}(t)\right\vert >2\pi^{2}n,\text{
}\forall k\neq\pm n.
\end{equation}
Similarly, for $t\in\lbrack\pi-h,\pi]$ and $n>N(h)$ there exist two
eigenvalues, counting multiplicity, denoted by $\lambda_{n}(t)$ and
$\lambda_{-n-1}(t)$ such that%
\begin{align}
\left\vert \lambda_{n}(t)-(2\pi n+t)^{2}\right\vert  &  \leq15\pi nh<n,\text{
}\left\vert \lambda_{-n-1}(t)-(2\pi n+t)^{2}\right\vert \leq15\pi nh<n,\\
\left\vert \lambda_{n}(t)-\lambda_{k}(t)\right\vert  &  >2\pi^{2}n,\text{
}\left\vert \lambda_{-n-1}(t)-\lambda_{k}(t)\right\vert >2\pi^{2}n,\text{
}\forall k\neq n,-(n+1).\nonumber
\end{align}

\end{summary}

In [27] we proved that the eigenvalues of $L_{t}$ can be numbered (counting
the multiplicity) by elements of $\mathbb{Z}$ such that, for each $n$ the
function $\lambda_{n}(t)$ is continuous on $[0,\pi]$ and for $|n|>N(h)$ the
inequalities (10)-(12) hold. Moreover, the bands $\Gamma_{n}$ of the spectrum
defined in (6) is constructed due to this numerations. Therefore inequalities
(11) and (12) and Proposition 1 yield the following.

\begin{proposition}
Suppose that $t\in\left(  \lbrack0,h]\cup\lbrack\pi-h,\pi]\right)  $ and
$|n|>N(h).$ Then

$(a)$ The multiplicity of the eigenvalues $\lambda_{n}(t)$ is not greater than
$2.$

$(b)$ If $\lambda_{n}(t)$ for $t\in\lbrack0,h]$ and $|n|>N(h)$ is the double
eigenvalue of $L_{t}$, then $\lambda_{n}(t)=\lambda_{-n}(t)$ and hence the
bands $\Gamma_{n}$ and $\Gamma_{-n}$ have the common point $\lambda_{n}(t)$.
Similarly, if $\lambda_{n}(t)$ for $t\in\lbrack\pi-h,\pi]$ and $|n|>N(h)$ is
the double eigenvalue of $L_{t}$, then $\lambda_{n}(t)=\lambda_{-n-1}(t)$ and
hence the bands $\Gamma_{n}$ and $\Gamma_{-n-1}$ have the common point
$\lambda_{n}(t)$.

$(c)$ If $n$ is a positive (negative) number, then the left and right ends of
$\Gamma_{n}$ are $\lambda_{n}(0)$ ($\lambda_{n}(\pi)$) and $\lambda_{n}(\pi)$
($\lambda_{n}(0)$)respectively.
\end{proposition}

Now using the above arguments we describe the shapes of $\Gamma_{n}$ for
$|n|>N(h).$

\begin{theorem}
Suppose that (2) holds and $|n|>N(h).$

$(a)$ $\lambda_{n}(t)$ for $t\in\lbrack h,\pi-h]$ are real and simple eigenvalues.

$(b)$ For $t\in\left(  \lbrack0,h]\cup\lbrack\pi-h,\pi]\right)  $ the double
eigenvalues $\lambda_{n}(t)$ are the real numbers.
\end{theorem}

\begin{proof}
$(a)$ By Summary 8, $\lambda_{n}(t)$ for $t\in\lbrack h,\pi-h]$ are the simple
eigenvalues. If $\lambda_{n}(t)\notin\mathbb{R}$ for some $t\in\lbrack
h,\pi-h]$, then $\overline{\lambda_{n}(t)}\neq\lambda_{n}(t)$ and by Summary 6
both are eigenvalues of $L_{t}(q)$ satisfying (10), which contradicts Summary 8.

$(b)$ If the double eigenvalue $\lambda_{n}(t)$ for $t\in\lbrack0,h]$ is
nonreal then arguing as in the proof of $(a)$ we conclude that there exist
four eigenvalues (counting the multiplicity) satisfying (11) which contradicts
Summary 9. In the same way the case $t\in\lbrack\pi-h,\pi]$ can be considered
\end{proof}

Thus the part $\left\{  \lambda_{n}(t):t\in\lbrack h,\pi-h]\right\}  $ of
$\Gamma_{n}$ is the interval of the real line. The following theorem shows
that, in fact, very large part of $\Gamma_{n}$ consists of the interval of the
real line.

\begin{theorem}
Suppose that (2) holds and $|n|>N(h).$

$(a)$ There may exists at most one number $\varepsilon_{n}$ and $\pi
-\delta_{n}$ in $[0,h)$ and $(\pi-h,\pi]$ respectively, such that $\lambda
_{n}(\varepsilon_{n})$ and $\lambda_{n}(\pi-\delta_{n})$ are double
eigenvalues. Then%
\begin{equation}
\lambda_{n}(\varepsilon_{n})=\lambda_{-n}(\varepsilon_{n})\in\mathbb{R},\text{
}\lambda_{n}(\pi-\delta_{n})=\lambda_{-n-1}(\pi-\delta_{n})\in\mathbb{R}%
,\text{ }\forall n>N(h).
\end{equation}

$(b)$ The eigenvalues $\lambda_{n}(t)$ for $t\neq\lbrack0,\pi]\backslash
\left\{  \varepsilon_{n},\pi-\delta_{n}\right\}  $ are simple and are not
complexation points. Moreover $\varepsilon_{n}\rightarrow0$ and $\delta
_{n}\rightarrow0$ as $n\rightarrow\infty.$

$(c)$ The double eigenvalues $\lambda_{n}(\varepsilon_{n})$ and $\lambda
_{n}(\pi-\delta_{n})$ are complexation points if and only if $\varepsilon
_{n}\neq0$ and $\delta_{n}\neq0$ respectively.

$(d)$ If $\lambda_{n}(\varepsilon_{n})$ and $\lambda_{n}(\pi-\delta_{n})$ are
complexation points, then the part $\left\{  \lambda_{n}(t):t\in
\lbrack\varepsilon_{n},\pi-\delta_{n}]\right\}  $ of $\Gamma_{n}$ is an
interval of the real line and is the real part $\operatorname{Re}(\Gamma_{n})$
of $\Gamma_{n}$ and the other parts $\gamma(0,\varepsilon_{n}):=\left\{
\lambda_{n}(t):t\in\lbrack0,\varepsilon_{n})\right\}  $ and $\gamma(\pi
,\delta_{n}):=\left\{  \lambda_{n}(t):t\in(\pi-\delta_{n},\pi]\right\}  $ are
not empty set and are the pure nonreal parts of $\Gamma_{n},$ that is, belong
to $\mathbb{C}\backslash\mathbb{R}$.
\end{theorem}

\begin{proof}
$(a)$ Suppose to the contrary that there exist two numbers $t_{1}$ and $t_{2}$
such that both $\lambda_{n}(t_{1})$ and $\lambda_{n}(t_{2})$ are the double
eigenvalues, where $0\leq t_{1}<t_{2}<h.$ Then by Proposition 2$(b)$ the bands
$\Gamma_{n}$ and $\Gamma_{-n}$ have two common points $\lambda_{n}(t_{1})$ and
$\lambda_{n}(t_{2})$ which contradicts Proposition 1$(b).$ The case
$(\pi-h,\pi]$ can be considered in the same way. The equalities and inclusions
in (13) follows from Proposition 2$(b)$ and Theorem 3$(b)$ respectively.

$(b)$ The simplicity of eigenvalues $\lambda_{n}(t)$ for $t\neq\lbrack
0,\pi]\backslash\left\{  \varepsilon_{n},\pi-\delta_{n}\right\}  $ follows
from Summary 8 and definitions of $\varepsilon_{n}$ and $\pi-\delta_{n}.$ Then
by Theorem 2 these eigenvalues are not complexation points. The relations
$\varepsilon_{n}\rightarrow0$ and $\delta_{n}\rightarrow0$ as $n\rightarrow
\infty$ are well-known. It readily follows from Summary 8 and from the
asymptotic formulas%
\begin{equation}
F(\lambda)=2\cos\sqrt{\lambda}+O(1/\sqrt{\lambda})
\end{equation}
as $\lambda\in\mathbb{R}$ and $\lambda\rightarrow\infty$ (see [12,Chap.1,
Sec.3] and [22, page 78]).

$(c)$ Theorem 2 implies that $\lambda_{n}(t)$ for $t=0,\pi$ are not
complexation points since by Proposition 2$(a)$ their multiplicities is not
greater than $2.$ Therefore it follows from $(b)$ that if $\lambda_{n}(t)$ is
a complexation point, then either $t=\varepsilon_{n}\neq0$ or $t=\pi
-\delta_{n}\neq\pi.$ Consider $\lambda_{n}(t)$ for $t=\varepsilon_{n}$ and
$\varepsilon_{n}\neq0.$ By Theorem 3$,$ $\lambda_{n}(\varepsilon_{n})$ and
$\lambda_{n}(h)$ are real numbers. Therefore by Theorem 1 the set $\left\{
\lambda_{n}(t):t\in\lbrack\varepsilon_{n},h]\right\}  $ $\subset\Gamma_{n}$ is
an interval of the real line. Since $\varepsilon_{n}\neq0$,$\pi,$ by Summary 7
any neighborhood of the double eigenvalue $\lambda_{n}(\varepsilon_{n})$
contains a nonreal number from $\sigma(L(q))$ . Thus, by Definition 1,
$\lambda_{n}(\varepsilon_{n})$ is a complexation point. In the same way we
prove that $\lambda_{n}(\pi-\delta_{n})$ is a complexation points.

$(d)$ By (13) and Theorem 1 the subset $\left\{  \lambda_{n}(t):t\in
\lbrack\varepsilon_{n},\pi-\delta_{n}]\right\}  $ of $\Gamma_{n}$ is an
interval of the real line. To prove the relation $\gamma(0,\varepsilon
_{n})\subset\mathbb{C}\backslash\mathbb{R}$ , first let us prove that
$\left\{  \lambda_{n}(t):t\in(0,\varepsilon_{n})\right\}  \subset
\mathbb{C}\backslash\mathbb{R}$. Suppose there exists $t_{0}\in(0,\varepsilon
_{n})$ such that $\lambda_{n}(t_{0})$ is a real number. Then by Theorem 1
$\left\{  \lambda_{n}(t):t\in\lbrack t_{0},\varepsilon_{n})\right\}
\subset\mathbb{R}$. It with the first statement of $(d)$ implies that there
exists a small neighborhood $U:=\left\{  \lambda\in\mathbb{C}:\left\vert
\lambda-\lambda_{n}(\varepsilon_{n})\right\vert <\delta\right\}  $\ of
$\lambda_{n}(\varepsilon_{n})$ which does not contain a nonreal number from
$\Gamma_{n}$. If $\lambda_{-n}(t)\in U\backslash\mathbb{R}$ for some
$t\in\lbrack0,h]$, then $\overline{\lambda_{-n}(t)}\neq\lambda_{-n}(t)$ and by
Summary 6 both are eigenvalues of $L_{t}(q)$ satisfying (11). Then there exist
three eigenvalues (counting the multiplicity) satisfying (11) which
contradicts Summary 9. Hence $U$ does not contain a nonreal number from
$\Gamma_{-n}$ too. Moreover, it readily follows from summaries 8 and 9 that
$U\cap\Gamma_{k}=\varnothing$ for all $k\neq n,-n.$ Thus $U$ does not contain
a nonreal number from $\sigma(L(q)).$ However it is impossible, since
$\lambda_{n}(\varepsilon_{n})$ is a complexation point. It remains to prove
that $\lambda_{n}(0)$ is not a real number. If we suppose to the contrary that
$\lambda_{n}(0)$ is a real number, then we obtain that the curve $\left\{
\lambda_{n}(t):t\in\lbrack0,\varepsilon_{n}]\right\}  $ with its symmetry
forms a closed curve in the spectrum, which contradicts Summary 5.
Thus\ $\gamma(0,\varepsilon_{n})\subset\mathbb{C}\backslash\mathbb{R}$. The
proof of $\gamma(\pi,\delta_{n})\subset\mathbb{C}\backslash\mathbb{R}$ is the same
\end{proof}

\begin{definition}
If $n>N(h),$ then the pure nonreal parts $\gamma(0,\varepsilon_{n})$ and
$\gamma(\pi,\delta_{n})$ of $\Gamma_{n}$ are said to be the left and right
tails of $\Gamma_{n}$, where $\gamma(0,\varepsilon_{n})$ and $\gamma
(\pi,\delta_{n})$ are defined in Theorem 4. If $n<-N(h)$ then we interchange
the words left and right.
\end{definition}

\begin{theorem}
Suppose that (2) holds and $n>N(h)$. Then the band $\Gamma_{n}$ has left
(right) nonreal tail and contains a left (right) complexation point
$\lambda_{n}(\varepsilon_{n})$ ($\lambda_{n}(\pi-\delta_{n})),$ where
$\varepsilon_{n}\in(0,h]$ ($(\pi-\delta_{n})\in\lbrack\pi-h,\pi),$ if and only
if $\lambda_{n}(0)$ ($\lambda_{n}(\pi))$ is a nonreal number. The theorem
continuous to hold if the condition $n>N(h)$ and the word left (right) are
replaced by $n<-N(h)$ and the word right (left).
\end{theorem}

\begin{proof}
We prove the theorem for $n>N(h)$ and $\lambda_{n}(0).$ The proofs for the
other cases are the same. If the eigenvalue $\lambda_{n}(0)$ is a real number
then by Theorem 1 $\left\{  \lambda_{n}(t):t\in\lbrack0,h]\right\}
\subset\mathbb{R},$ since $\lambda_{n}(h)$ is also real number (see Theorem
3($a$)). It means that the left tail absent and by Theorem 4$(d)$ the left
complexation point absent. Now suppose that $\lambda_{n}(0)$ is a nonreal
number. \ Denote by $t_{0}$ the smallest number in $[0,h]$ such that
$\lambda_{n}(t_{0})$ is real. Since $\lambda_{n}(0)$ is nonreal number,
$\lambda_{n}(h)$ is real number and $\lambda_{n}$ is a continuous function, we
have $t_{0}\in(0,h]$. Then $\left\{  \lambda_{n}(t):t\in\lbrack t_{0}%
,\pi-h]\right\}  \subset\mathbb{R}$ and any neighborhood of $\lambda_{n}%
(t_{0})$ contains a nonreal number $\lambda_{n}(t)$, where $t\in\lbrack
0,t_{0}).$ It means that $\lambda_{n}(t_{0})$ is a complexation point and
$\left\{  \lambda_{n}(t):t\in\lbrack0,t_{0})\right\}  $ is the left tail of
$\Gamma_{n}$
\end{proof}

Now we find the necessary and sufficient conditions in which $\sigma(L(q))$ is
a half line.

\begin{theorem}
Suppose that (2) holds. Then the spectrum of $L(q)$ is a half line if and only
if the followings hold:

$\left(  i\right)  $ one eigenvalue of $L_{0}(q)$ is simple and the all others
are double,

$(ii)$ all eigenvalues of $L_{\pi}(q)$ are double.
\end{theorem}

\begin{proof}
First suppose that $\sigma(L(q))$ $=[a,\infty)$ for same $a$. Then the bands
$\Gamma_{n}$ of the spectrum $\sigma(L(q))$ are the intervals of the real line
and the left end of the first interval is $a.$ Since these intervals may have
at most one common points (see Proposition 1$(b)$), that is, do not overlap
each other and $a$ is the leftmost point of the spectrum it is the end only of
one band of the spectrum. Therefore by Remark 1 it is a simple eigenvalue.
Moreover, using the asymptotic formulas (14) for $F(\lambda)$ and taking into
account that $F(\lambda)\in\mathbb{R}$ for $\lambda\in\mathbb{R}$ (see Summary
7) we see that the graph $\left\{  (\lambda,F(\lambda)):\lambda\in
\mathbb{R})\right\}  $ of $F(\lambda)$ intersect the line $y=2$ before the
line $y=-2.$ Therefore $a$ is the simple eigenvalue of $L_{0}(q).$ Thus by
Proposition 1$(b)$ the right end point $b$ of the first band is the eigenvalue
of $L_{\pi}(q).$ Since $\sigma(L(q))$ has no gaps, $b$ is the left end point
of the second interval. Therefore $b$ is the double eigenvalue of $L_{\pi
}(q).$ The right end point of the second interval is the eigenvalue $c$ of
$L_{0}(q).$ Since $\sigma(L(q))$ has no gaps, $c$ is the left end point of the
third interval and hence $c$ is the double eigenvalue of $L_{0}(q).$
Continuing this procedure we obtain that all eigenvalues of $L_{0}(q)$ except
$a$ and all eigenvalues of $L_{\pi}(q)$ are double.

Now suppose that $(i)$ and $(ii)$ holds. Let $\lambda_{k}(0)$ be the simple
eigenvalue of $L_{0}(q).$ Then by Remark 1 it is end point of only one band
$\Gamma_{k}$. If $\lambda_{k}(0)$ is nonreal then by Summary 6, $\overline
{\lambda_{k}(0)}$ is also an eigenvalue of $L_{0}(q)$ and is end point of only
one band. It implies that $\overline{\lambda_{k}(0)}$ is also a simple
eigenvalue of $L_{0}(q)$ which contradict to $(i).$ Thus $\lambda_{k}(0)$ is real.

Now consider the connectedness of $\sigma(L(q)).$ Let $m>\max\left\{
\left\vert k\right\vert ,N(h)\right\}  .$ Then it readily follows from
Proposition 2$(b)$ and $(i),$ $(ii)$ that%
\begin{equation}
\lambda_{m}(\pi)=\lambda_{-m-1}(\pi),\text{ }\lambda_{-n}(0)=\lambda_{n}\text{
}(0),\text{ }\lambda_{n}(\pi)=\lambda_{-n-1}(\pi),\text{ }\forall n>m.
\end{equation}
It means that $\Gamma:=%
{\textstyle\bigcup\nolimits_{n:\left\vert n\right\vert >m}}
\Gamma_{n}$ is a connected curve. Consider $\gamma=%
{\textstyle\bigcup\nolimits_{n:\left\vert n\right\vert \leq m}}
\Gamma_{n}.$ Suppose that $\gamma$ is not a connected curve, that is, there
exist at least two separated curves $\gamma_{1}$ and $\gamma_{2}$ lying in
$\gamma.$ Then $\gamma_{1}$ and $\gamma_{2}$ have $\ 4$ end points denoted by
$a_{1},$ $a_{2},$ $a_{3}$ and $a_{4}.$ One of them, say $a_{1}$ is
$\lambda_{k}(0).$ It readily follows from Proposition 2$(b)$ that one of
others, say $a_{4},$ is $\lambda_{m}(\pi)\in\Gamma$ and $a_{i}$ for $i=2,3$ do
not belong to $\Gamma.$ It is also clear that there exist $p$ such that
$a_{2}$ is the end point of $\Gamma_{p}$ and is not end point of $\Gamma_{s}$
for $s\neq p.$ It, by Remark 1, implies that $a_{2}$ is a simple eigenvalue of
either $L_{0}(q)$ or $L_{\pi}(q).$ It contradicts $(i)$ and $(ii)$, since
$a_{2}\neq a_{1}.$ Thus we proved that $\gamma$ is a connected curve with end
points \ $a_{1}$ and $a_{4}=\lambda_{m}(\pi)\in\Gamma.$ Therefore
$\sigma(L(q))$ is a connected curve with one left end point $a_{1}%
\in\mathbb{R}$. On the other hand, by Theorem 3$(a)$ we have $\lambda
_{n}(h)\in\mathbb{R}$ for all $n>m.$ Therefore repeating the proof of Theorem
1$(a),$ we see that the part of \ $\sigma(L(q))$ from $a_{1}\in\mathbb{R}$ to
$\lambda_{n}(h)\in\mathbb{R}$, where $n>m,$ is the interval \ $\left[
a_{1},\lambda_{n}(h)\right]  .$ Letting $n$ tend to infinity we get
$\sigma(L(q))=[a_{1},\infty)$
\end{proof}

\section{Reality and non-reality of $\Gamma_{n}$ for large $n$}

To consider the reality of $\Gamma_{n}$ for large $n$ we use the following theorem.

\begin{theorem}
Suppose that (2) holds and $\left\vert n\right\vert >N(h)$. Then the
followings are equivalent

$(a)$ Band $\Gamma_{n}$ is the real interval.

$(b)$ The eigenvalues $\lambda_{n}(0)$ and $\lambda_{n}(\pi)$ are the real numbers.

$(c)$ $\lambda_{n}(t)$ for $t\in(0,\pi)$ are simple eigenvalues of $L_{t}(q).$

$(d)$ The eigenvalues $\lambda_{n}(t)$ for $t\in(0,\pi)$ are not the
complexation points.
\end{theorem}

\begin{proof}
By Proposition 1$(a)$ and Theorem 1$(b)$ the statements $(a)$ and $(b)$ are
equivalent. It readily follows from Theorem 4 that $(c)$ and $(d)$ are
equivalent. Theorem 4$(c)$ and Theorem 5 imply that $(b)$ and $\ (d)$ are equivalent.
\end{proof}

Now we find the conditions on the Fourier coefficients of the PT-symmetric
potential $q$ for reality and non-reality of $\Gamma_{n}$ for large $n$. For
this first of all we need to consider the Fourier coefficient of the
potentials (2). In general the locally integrable periodic function has no
Fourier decomposition. However we can consider the Fourier coefficients%
\[
q_{n}:=(q,e^{i2\pi nx}):=\int_{0}^{1}q(x)e^{-i2\pi nx}dx
\]
for $n\in\mathbb{Z}$. It is clear and well-known that if $q$ is PT-symmerric,
then using the substitution $t=-x$ one can get the equality
\begin{equation}
\overline{q_{n}}=\int_{0}^{1}\overline{q(x)}e^{i2\pi nx}dx=\int_{0}%
^{1}\overline{q(-t)}e^{-i2\pi nt}dt=\int_{0}^{1}q(t)e^{-i2\pi nt}dt=q_{n}%
\end{equation}
which means that $q_{n}\in\mathbb{R}$ for all $n\in\mathbb{Z}$.

First to find the necessary and sufficient conditions on the potential for
reality and non-reality of $\Gamma_{n}$ we construct some class of periodic
PT-symmetric function as follows. Let $S_{p}$ be the set of $1$ periodic
PT-symmetric functions $q\in W_{1}^{p}[0,1]$\textit{ }such that\textit{ }%
\begin{equation}
q(1)=q(0),\text{ }q^{\prime}(1)=q^{\prime}(0),...,\text{ }q^{(s-1)}%
(1)=q^{(s-1)}(0)\
\end{equation}
\textit{ }for some $s\leq p$ and there exist positive constants $c_{1},$
$c_{2},$ $c_{3}$ and $N>0$ satisfying%

\begin{equation}
\mid q_{n}\mid>c_{1}n^{-s-1}\And c_{2}\text{ }\left\vert q_{n}\right\vert
\leq\left\vert q_{-n}\right\vert \leq c_{3}\left\vert q_{n}\right\vert ,\text{
}\forall n>N.
\end{equation}
In particular, $S_{0}$ is the set of $1$ periodic PT-symmetric functions $q\in
L_{1}[0,1]$\textit{ }satisfying%

\[
\mid q_{n}\mid>c_{1}n^{-1}\And c_{2}\text{ }\left\vert q_{n}\right\vert
\leq\left\vert q_{-n}\right\vert \leq c_{3}\left\vert q_{n}\right\vert ,\text{
}\forall n>N.
\]

Besides we use the following formulas obtained in [23] (see (40)-(44) of
[23]). In [23] we proved that $\lambda_{n}(t)$ satisfies the equations
\begin{equation}
(\lambda-(2\pi n)^{2}-t^{2}-\frac{1}{2}(A(\lambda,t)+A^{^{\prime}}%
(\lambda,t)))^{2}=D(\lambda,t),
\end{equation}
where $D=(4\pi nt)^{2}+q_{2n}q_{-2n}+8\pi ntC+C^{2}+q_{2n}B^{^{\prime}%
}+q_{-2n}B+BB^{^{\prime}},$ $C=\frac{1}{2}\left(  A-A^{\prime}\right)  $,
\begin{equation}
A(\lambda,t)=\sum_{k=1}^{\infty}a_{k}(\lambda,t),\text{ }A^{\prime}%
(\lambda,t)=\sum_{k=1}^{\infty}a_{k}^{\prime}(\lambda,t),
\end{equation}%
\begin{equation}
a_{k}(\lambda,t)=\sum_{n_{1},n_{2},...,n_{k}}q_{-n_{1}-n_{2}-...-n_{k}}%
{\textstyle\prod\limits_{s=1}^{k}}
q_{n_{s}}\left(  \lambda-(2\pi(n-n_{1}-..-n_{s})+t)^{2}\right)  ^{-1},
\end{equation}%
\begin{equation}
a_{k}^{^{\prime}}(\lambda,t)=\sum_{n_{1},n_{2},...,n_{k}}q_{-n_{1}%
-n_{2}-...-n_{k}}%
{\textstyle\prod\limits_{s=1}^{k}}
q_{n_{s}}\left(  \lambda-(2\pi(n+n_{1}+..+n_{s})-t)^{2}\right)  ^{-1}.
\end{equation}
The functions $B$ and $B^{^{\prime}}$ are obtained respectively from $A$ and
$A^{^{\prime}}$ by replacing $q_{-n_{1}-n_{2}-...-n_{k}}$ with $q_{2n-n_{1}%
-n_{2}-...-n_{k}}$ and $q_{-2n-n_{1}-n_{2}-...-n_{k}}$. One can readily see
from (20)-(22) that
\begin{equation}
\left\{  A(\lambda,t),A^{^{\prime}}(\lambda,t),C(\lambda,t),\ B(\lambda
,t),B^{^{\prime}}(\lambda,t),\ D(\lambda,t)\right\}  \in\mathbb{R}%
\end{equation}
if $\lambda,t$ and the Fourier coefficients $q_{n}$ for $n\in\mathbb{Z}$ are
real number. Moreover, there exists a constant $K$ such that the following
inequalities hold (see (56) of [23])%
\begin{equation}
\mid A(\lambda,t)+A^{^{\prime}}(\lambda,t)-A(\mu,t)-A^{^{\prime}}(\mu
,t)\mid<Kn^{-2}\mid\lambda-\mu\mid.
\end{equation}
Besides the equalities
\begin{equation}
C(\lambda_{n,j}(t),t)=tO(n^{-1}),\text{ }B(\lambda_{n,}(t),t)=o\left(
n^{-s-1}\right)  ,\text{ }B^{^{\prime}}(\lambda_{n,}(t),t)=o\left(
n^{-s-1}\right)
\end{equation}
hold uniformly with respect to $t$ in $[0,h]$ (see (57) and (46) of [23]).

\begin{remark}
Formula (19) is proved in \ [23] under conditions (17) and (18) (see Theorem
2.4). In [25] we proved ([see Theorem 2 and (35), (37) of [25]]) that
$\lambda_{n}(t)$ satisfies (19) without conditions (17) and (18) which readily
follows from (37) and (38) of [23] (see the first paragraph of the proof of
Theorem 2 of [25]).
\end{remark}

\begin{theorem}
Suppose that (2) holds, $n$ is a large number and $t\in\left(  \lbrack
0,h]\cup\lbrack\pi-h,\pi]\right)  .$ Then the eigenvalue $\lambda_{n}(t)$ is
real if and only if
\begin{equation}
D(\lambda_{n}(t),t)\geq0.
\end{equation}

\end{theorem}

\begin{proof}
We prove the theorem $t\in\lbrack0,h].$ The proof for $t\in\lbrack\pi-h,\pi]$
is the same. Suppose that $\lambda_{n}(t)$ is real. Since the Fourier
coefficients of the PT-symmetric function $q$ are real numbers (see (16)), by
(23) we have
\begin{equation}
F(\lambda_{n}(t),t):=\left(  \lambda_{n}(t)-(2\pi n)^{2}-t^{2}-1/2(A(\lambda
_{n}(t),t)+A^{^{\prime}}(\lambda_{n}(t),t))\right)  \in\mathbb{R}.
\end{equation}
Therefore the right side of (19) for $\lambda=\lambda_{n}(t)$ are nonnegative,
that is, (26) holds.

Now suppose that (26) holds. Then, by (19), the relation (27) holds too. It
remains to show that (27) implies the reality $\lambda_{n}(t).$ Suppose that
$\lambda_{n}(t)$ is nonreal. Since $F(\lambda_{n}(t),t)$ is real number (see
(27)), from (20)-(22) one can readily see that $F(\lambda_{n}(t),t)=F\left(
\overline{\lambda_{n}(t)},t\right)  $ from which by using (24) we obtain the
following contradiction
\begin{align*}
2\left\vert \lambda_{n}(t)-\overline{\lambda_{n}(t)}\right\vert  &
=\left\vert A(\lambda_{n}(t),t)+A(\overline{\lambda_{n}(t)},t)-A^{^{\prime}%
}(\lambda_{n}(t),t))-A^{^{\prime}}(\overline{\lambda_{n}(t)},t)\right\vert \\
&  <2Kn^{-2}\mid\lambda_{n}(t)-\overline{\lambda_{n}(t)}\mid.
\end{align*}
The theorem is proved
\end{proof}

The following corollary immediately follows from theorems 3, 7 and 8.

\begin{corollary}
Suppose that (2) holds and $n$ is a large number. Then $\Gamma_{n}%
\subset\mathbb{R}$ if and only if
\begin{equation}
D(\lambda_{n}(0),0)\geq0\text{ }\And\text{ }D(\lambda_{n}(\pi),\pi)\geq0.
\end{equation}

\end{corollary}

Now using this corollary, (18) and (25) we prove the following.

\begin{theorem}
Suppose that $q\in S_{p}$ and $n$ is a large number. Then $\Gamma_{n}%
\subset\mathbb{R}$ if and only if
\begin{equation}
q_{n}q_{-n}>0.
\end{equation}

\end{theorem}

\begin{proof}
It follows from (25) that
\begin{align}
C(\lambda_{n}(0),0)  &  =0,\text{ }D(\lambda_{n}(0),0)=q_{2n}q_{-2n}%
+q_{2n}B^{^{\prime}}+q_{-2n}B+BB^{^{\prime}},\\
D(\lambda_{n}(0),0)  &  =q_{2n}q_{-2n}+o\left(  q_{-2n}n^{-s-1}\right)
+o\left(  q_{2n}n^{-s-1}\right)  +o\left(  n^{-2s-2}\right)  .\nonumber
\end{align}
This with (18) implies that
\begin{equation}
D(\lambda_{n}(0),0)=q_{2n}q_{-2n}(1+o\left(  1\right)  )
\end{equation}
as $n\rightarrow\infty.$ Now suppose that (29) holds. Then by Theorem 2.12 of
[23] the eigenvalue $\lambda_{n}(t)$ is simple for all $t\in\lbrack0,\pi].$
Therefore, by Theorem 7, $\Gamma_{n}$ is real.

Now suppose that $\Gamma_{n}$ is real. Then by Corollary 1, (28) holds. The
first equality in (28) with (31) imply that $q_{2n}q_{-2n}>0,$ since
$q_{2n}q_{-2n}$ is real (see (16)) and nonzero (see (18)). In the same way
using the second equality in (28) we obtain $q_{2n+1}q_{-2n-1}>0$ \ 
\end{proof}

In [18] we proved that if $q\in W_{1}^{s}[0,1]$ and (17) holds then%
\begin{align}
B(\lambda_{n}(0),0) &  =-S_{2n}+2Q_{0}Q_{2n}+o\left(  n^{-s-2}\right)
,\quad\nonumber\\
B^{\prime}(\lambda_{n}(0),0) &  =-S_{-2n}+2Q_{0}Q_{-2n}+o\left(
n^{-s-2}\right)  ,
\end{align}
where $Q_{k}\ $and$\,S_{k}$ are the Fourier coefficients of the function $Q$
and $S$ defined by
\[
Q(x)=\int_{0}^{x}q(t)\,dt,\quad S(x)=Q^{2}(x)
\]
(see Lemma 6 of [18] ) and $-S_{\pm2n}+2Q_{0}Q_{\pm2n}$ are real numbers (see
page 655). Moreover
\begin{equation}
q_{n}=o\left(  n^{-s}\right)  ,\text{ }S_{\pm2n}=o\left(  n^{-s-1}\right)
,\text{ }Q_{\pm2n}=o\left(  n^{-s-1}\right)
\end{equation}
(see page 658 of [18] ). Therefore using (30), (32) and (33) we obtain
\[
D(\lambda_{n}(0),0)=P_{2n}+o\left(  n^{-2s-2}\right)  ,
\]
where $P_{n}=q_{n}q_{-n}-q_{n}\left(  S_{-n}-2Q_{0}Q_{-n}\right)
-q_{-n}(S_{n}-2Q_{0}Q_{n}).$ Similarly
\[
D(\lambda_{n}(\pi),\pi)=P_{2n+1}+o\left(  n^{-2s-2}\right)  ,
\]
If there exist $c$ such that%
\begin{equation}
\left\vert P_{n}\right\vert >cn^{-2s-2},
\end{equation}
then we have
\begin{equation}
D(\lambda_{n}(0),0)=P_{2n}(1+o\left(  1\right)  ),\text{ }D(\lambda_{n}%
(\pi),\pi)=P_{2n+1}(1+o\left(  1\right)  ).
\end{equation}
Therefore using (35), Theorem 8 and Theorem 5 and taking into account that if
$q\in W_{1}^{s}[0,1]$ and (2) holds then (17) holds too we obtain

\begin{theorem}
If $q\in W_{1}^{s}[0,1],$ (2) and (34) hold, $n$ is a large number and
$P_{n}<0$ then $\Gamma_{n}$ has the nonreal tails.
\end{theorem}

Now let us consider the connections between the reality of $\sigma(L(q))$ and
the spectrality of $L(q)$ for $q\in S_{p}.$ In [13] it was proved that
$L(q)$\textit{\ }is a spectral operator if and only if\textit{ }%
\begin{equation}
\sup_{\gamma\in R}(\sup_{t\in(-\pi,\pi]}\parallel e(t,\gamma)\parallel
)<\infty,
\end{equation}
where $e(t,\gamma)$ be the spectral projection defined by contour integration
of the resolvent of $L_{t}(q)$, $\gamma\in R$ and $R$ is the ring consisting
of all sets which are the finite union of the half closed rectangles (see
Theorem 3.5). Note that the spectral singularities of the operator $L(q)$ are
the points of its spectrum in neighborhoods of which the projections of
$e(t,\gamma)$ are not uniformly bounded (see [9] and [21-23]). Therefore if
$L(q)$ has a spectral singularity then it is not spectral operator. However
may be $\sigma(L(q))$ does not contain a spectral singularity nevertheless
$L(q)$ is not a spectral operator. It happens if $\parallel e(t,\gamma
)\parallel\rightarrow\infty$ as $\gamma$ goes to infinity. In this case we say
that $L(q)$ has a spectral singularity at infinity (see Definition 3.2 in
[23]). According to (36), we say that $L(q)$ as an asymptotically spectral
operator if the inequality obtained from (36) by replacing $R$ with $R(C)$
holds, where $C$ is a large positive number and $R(C)$ is the ring consisting
of all sets which are the finite union of the half closed rectangles lying in
$\{\lambda\in\mathbb{C}:\mid\lambda\mid>C\}$ (see Definition 3.6 in [23]). If
$L(q)$ is an asymptotically spectral operator then it has noncomplicated
spectral expansion (see Theorem 4 in [27]).

\begin{theorem}
$(a)$ Let $\left\vert n\right\vert >N(h)$ and $a\in(0,\pi).$ Then the
followings are equivalent.

$1)$ $\lambda_{n}(a)$ is a spectral singularity of $L(q)$.

$2)$ $\lambda_{n}(a)$ is a complexation point of $\sigma(L(q)).$

$3)$ $\lambda_{n}(a)$ is a multiple eigenvalue of $L_{a}(q).$

$4)$ If $a\in(0,h]$ then $\left\{  \lambda_{n}(t):t\in\lbrack0,a)\right\}  ;$
if $a\in\lbrack\pi-h,\pi)$ then $\left\{  \lambda_{n}(t):t\in(a,\pi]\right\}
\ $ is a nonreal tail of $\Gamma_{n}.$

$(b)$ Let $q\in S_{p}$ for some $p=0,1,...$ Then the followings are equivalent.

$1)$ $L(q)$ is an asymptotically spectral operator.

$2)$ There exists $m\geq N(h)$ such that (29) holds for all $\left\vert
n\right\vert >m$.

$3)$ There exist $m\geq N(h)$ such that $\Gamma_{n}$ for $\left\vert
n\right\vert >m$ are real pairwise disjoint intervals separated by the gaps of
the spectrum.
\end{theorem}

\begin{proof}
$(a)$ By Theorem 4, $2),$ $3)$ and $4)$ are equivalent. On the other hand, by
Proposition 2 of [26] $\lambda_{n}(a)$ is a spectral singularity if and only
if it a multiple eigenvalue.

$(b)$ First let us show that $2)$ and $3)$ are equivalent. By Theorem 9, $3)$
implies $2)$. Now suppose that $2)$ holds. Then again by Theorem 9,
$\Gamma_{n}$\ are real for all $\left\vert n\right\vert >m.$ On the other hand
if (29) holds then by Theorem 2.12 of [23] the eigenvalues $\lambda_{n}(t)$
for $t\in\lbrack0,\pi]$ are simple and hence by by Remark 1, $\Gamma_{n}$ for
$\left\vert n\right\vert >m$ are real pairwise disjoint intervals, that is
$3)$ holds. In [23] we proved that if (29) holds then $L(q)$ is asymptotically
spectral operator, that is, if $2)$ holds then $1)$ holds too. Now suppose
that $2)$ does not hold. Then by Theorem 9 there exists a sequence $\left\{
n_{k}\right\}  $ such that $\Gamma_{n_{k}}$ is not a real interval that is
contains a nonreal tail. Then by $(a)$ there exist $t_{n_{k}}\in(0,\pi)$ such
that $\lambda_{n_{k}}(t_{n_{k}})$ is a spectral singularity. \ It means that
the conditions of Theorem 1$(c)$ of [27] does not hold and hence $L(q)$ is not
an asymptotically spectral operator. Thus $2)$ and $1)$ are equivalent
\end{proof}

Now we consider a simple example that helps to see the complexity of the
relations between the spectrum and spectrality of the PT-symmetric periodic operators.

\begin{example}
The operators $L_{t}(q)$ and $L(q)$ are denoted by $H_{t}(a,b)$ and $H(a,b)$
when
\begin{equation}
\text{ }q(x)=ae^{-i2\pi x}+be^{i2\pi x},
\end{equation}
where $a$ and $b$ are the real numbers, that is, $q$ is a PT-symmetric potential.
\end{example}

To consider the spectrum of these operators we use the following results of
[24] formulated as Summary 10 (see Theorem 1 and (26) of [24]).

\begin{summary}
If $ab=cd$, then $\sigma(H(a,b))=\sigma(H(c,d))$ and $\sigma(H_{t}%
(a,b))=\sigma(H_{t}(c,d))$. The operators $H_{t}(a,b))$ and $H_{t}(c,d)$ have
the same characteristic equation (4).
\end{summary}

By this summary if $ab>0$ then the spectra of $H(c,d)$ and $H_{t}(a,b)$
coincide with the spectrum of the self adjoint operators $L(2c\cos2\pi x)$ and
$L_{t}(2c\cos2\pi x)$ respectively, where $c$ is a positive square root of
$ab.$ It is well-known that (see Chapter 21 of [19] and Chapter 2 of [7]) all
eigenvalues of $L_{t}(2c\cos2\pi x)$ for all $t\in\lbrack0,2\pi)$ are real and
simple and all gaps in the spectrum of $L(2c\cos2\pi x)$ are open. Therefore
we have the following.

\begin{theorem}
If $ab>0$ then all eigenvalues of $H_{t}(a,b)$ for all $t\in\lbrack0,2\pi)$
are real and simple and the spectrum of $H(a,b)$ consist of the real intervals
$\Gamma_{n}$ separated by gaps, where $H_{t}(a,b)$ and $H(a,b)$ are defined in
Example 1.
\end{theorem}

To consider the spectrality of $H(a,b)$ we use the following results of [25]
and [27] formulated as Summary 11 (see Prop. 3 and Theorem 6 of [25] and
Corollary 1 of [27]).

\begin{summary}
$(a)$ If $\mid a\mid\neq\mid b\mid,$ then the operator $H(a,b)$ has the
spectral singularity at infinity and hence is not an asymptotically spectral operator.

$(b)$ \textit{ }The operator $H(a,b)$ is an asymptotically spectral operator
and has no spectral singularity at infinity if and only if
\[
\mid a\mid=\mid b\mid,\text{ }\inf_{q,p\in\mathbb{N}}\{\mid q\alpha
-(2p-1)\mid\}\neq0,
\]
where $\alpha=\pi^{-1}\arg(ab).$
\end{summary}

This summary immediately yields the following.

\begin{theorem}
The PT-symmetric operator $H(a,b)$ is a spectral operator if and only if
$a=b,$ that is, (37) is the real potential $2a\cos2\pi x$
\end{theorem}

\begin{proof}
If $a=b$ then the operator $H(a,b)$ is the self-adjoint operator $L(2a\cos2\pi
x)$ and hence is the spectral operator. Now suppose that $a\neq b.$ Since $a$
and $b$ are the real numbers, it is possible in the following cases. Case 1:
$\mid a\mid\neq\mid b\mid$ and Case 2: $b=-a.$ In Case 1 by Summary 11$(a),$
the operator $H(a,b)$ has the spectral singularity at infinity and hence is
not a spectral operator. \ In Case 2 we have $\pi^{-1}\arg(ab)=1.$ Therefore,
by Summary 11$(b),$ $H(a,b)$ is not a spectral operator
\end{proof}

Now using Theorem 4 and Theorem 5 we prove the following.

\begin{theorem}
Suppose that (2) holds. If there exists $m>N(h)$ such that $\lambda_{n}(0)$
and $\lambda_{n}(\pi)$ for $n>m$ are nonreal numbers then there exists $R$
such that $[R,\infty)\subset\sigma(L(q))$ and the number of gaps in the real
part $\operatorname{Re}(\sigma(L(q)))$ of $\sigma(L(q))$ is finite.
\end{theorem}

\begin{proof}
Let and $n>m+1>N(h)+1.$ By Theorem 5 there exist complexation points
$\lambda_{n}(\varepsilon_{n})$ and $\lambda_{n}(\pi-\delta_{n})$, where
$\varepsilon_{n}\in(0,h],$ $(\pi-\delta_{n})\in\lbrack\pi-h,\pi).$ Then using
Theorem 4$(d)$ and taking into account that $\Gamma_{n}$ is a simple open
curve (see Proposition 1$(b)$) and $\lambda_{n}(\varepsilon_{n})<\lambda
_{n}(\pi-\delta_{n})$ (see (11) and (12)) we obtain%
\[
\operatorname{Re}(\Gamma_{n})=\left\{  \lambda_{n}(t):t\in\lbrack
\varepsilon_{n},\pi-\delta_{n}]\right\}  =\left[  \lambda_{n}(\varepsilon
_{n}),\lambda_{n}(\pi-\delta_{n})\right]  .
\]
It follows from Summary 6 and Summary 9 that $\lambda_{-n}(0)$ is also nonreal
number. Therefore in the same way and using (13) we prove that the right end
point of $\operatorname{Re}(\Gamma_{-n})$ is $\lambda_{-n}(\varepsilon_{n})$
and coincides with the left ent point of $\operatorname{Re}(\Gamma_{n}).$
Instead of $\lambda_{n}(0)$ using $\lambda_{n}(\pi)$ and repeating this proof
we obtain that the right end point of $\operatorname{Re}(\Gamma_{n})$
coincides with the left ent point of $\operatorname{Re}(\Gamma_{-n-1}).$
Therefore there exist $R$ such that $[R,\infty)\subset\sigma(L(q)).$ To
complete the proof it is enough to note that by Theorem 1$(a),$
$\operatorname{Re}\left(  \Gamma_{n}\right)  $ for each $n\in\mathbb{Z}$ is
either empty set or a point or an interval.
\end{proof}

Now we find the conditions on the \ Fourier coefficients%
\begin{equation}
\text{ }f_{n}=\int_{0}^{1}f(x)\cos2\pi nxdx,\text{ }g_{n}=\int_{0}^{1}%
g(x)\sin2\pi nxdx\text{ }%
\end{equation}
for which the number of gaps in the real part of the spectrum of $L(q)$ is
finite, where $f=\operatorname{Re}q$ and $g=\operatorname{Im}q.$ It is clear
that if $q$ is PT-symmerric, then $f$ and $g$ are even and odd functions
respectively and hence
\begin{equation}
q_{n}=f_{n}+g_{n},\text{ }q_{-n}=f_{n}-g_{n}.
\end{equation}
Using (34), (35), Theorem 8 and Theorem 14 and taking into account that if
$q\in W_{1}^{s}[0,1]$ and (2) holds then (17) holds too we obtain we obtain.

\begin{theorem}
Suppose that (2) holds and $q\in W_{1}^{s}[0,1]$ for some $s\geq0.$ If there
exist positive constants $m$ and $\alpha$ such that for $n>m$ the inequality
\begin{equation}
P_{n}<-\alpha n^{-2s-2}%
\end{equation}
holds, then the number of gaps in the real part of the spectrum of $L(q)$ is finite.
\end{theorem}

Now we prove the other and more applicable theorem.

\begin{theorem}
Suppose that (2) holds and $q\in W_{1}^{s}[0,1]$ for some $s\geq0.$ If there
exist $\delta>1$, $\beta>0$\ and $m>0$ such that
\begin{equation}
\left\vert g_{n}\right\vert >\beta n^{-s-1},\text{ }\left\vert g_{n}%
\right\vert >\delta\left\vert f_{n}\right\vert
\end{equation}
\ for all $n>m,$ then the number of gaps in the real part of the spectrum of
$L(q)$ is finite, where $f_{n}$ and $g_{n}$ are the Fourier coefficients of
$\operatorname{Re}q$ and $\operatorname{Im}q$ defined in (38).
\end{theorem}

\begin{proof}
It readily follows from (39) and (41) that%
\[
\left\vert q_{n}\right\vert >\varepsilon n^{-s-1},\text{ \ }\left\vert
q_{-n}\right\vert >\varepsilon n^{-s-1}%
\]
and hence
\begin{equation}
-q_{n}q_{-n}=\left(  g_{n}^{2}-f_{n}^{2}\right)  >\gamma n^{-2s-2},\text{
}\frac{\left\vert q_{\pm n}\right\vert }{\left\vert q_{n}q_{-n}\right\vert
}\text{ }=O(n^{s+1})
\end{equation}
for some $\varepsilon>0$ and $\gamma>0.$ Using (2) and taking into account
that $q^{(s)}$ is integrable on $[0,1]$ one can readily verify that the
functions $Q$ and $S$ are $1$ periodic functions, $Q^{(s+1)}$ and $S^{(s+1)}$
are integrable on $[0,1]$ and hence $S_{\pm n}=o\left(  n^{-s-1}\right)  $ and
$Q_{\pm n}=o\left(  n^{-s-1}\right)  $. Therefore, using (42) and the
definition of $P_{n}$ we obtain%
\[
P_{n}=q_{n}q_{-n}-q_{n}\left(  S_{-n}-2Q_{0}Q_{-n}\right)  -q_{-n}%
(S_{n}-2Q_{0}Q_{n})=\left(  f_{n}^{2}-g_{n}^{2}\right)  (1+o(1)).
\]
Thus if (41) holds then (40) holds too. Hence the proof follows from Theorem 15
\end{proof}

Theorem 16 shows that one can construct the large and easily checkable classes
of the finite-zone PT-symmetric periodic potentials, by constructing the
potentials $q$ so that the Fourier coefficients of $g=\operatorname{Im}q$ is
greater (by absolute value) than the Fourier coefficient of
$f=\operatorname{Re}q.$ One of them are given in the next theorem.

\begin{theorem}
Suppose that (2) holds, $g\in W_{1}^{s+1}(a,a+1)\cap W_{1}^{s}[0,1]$ for some
$s\geq0$ and $a\in\lbrack0,1)$ and $g^{(s)}$ has a jump discontinuity at $a$
with size of the jump $c:=g^{(s)}(a+0)-g^{(s)}(a-0),$ while either $f\in
W_{1}^{s+1}[0,1]$ or $f\in W_{1}^{s+1}(b,b+1)\cap W_{1}^{s}[0,1]$ for some
$b\in\lbrack0,1)$ and $f^{(s)}$ has a jump discontinuity at $b$ with size of
the jump $d:=f^{(s)}(b+0)-f^{(s)}(b-0).$ If $\left\vert d\right\vert
<\left\vert c\right\vert $ then the number of gaps in the real part of the
spectrum of $L(q)$ is finite.
\end{theorem}

\begin{proof}
The Fourier coefficient $g_{n}$ defined in (38) can be calculated by%
\[
g_{n}=\lim_{\varepsilon\rightarrow0}\int_{a+\varepsilon}^{a+1-\varepsilon
}g(x)\sin2\pi nxdx.
\]
Applying $s+1$ times the integration by parts formula we get
\[
\left\vert g_{n}\right\vert =\left\vert c\right\vert (2\pi n)^{-s-1}%
+o(n^{-s-1}).
\]
In the same way we obtain that either $f_{n}=o(n^{-s-1})$ or
\[
\left\vert f_{n}\right\vert =\left\vert d\right\vert (2\pi n)^{-s-1}%
+o(n^{-s-1}).
\]
Therefore if $\left\vert d\right\vert <\left\vert c\right\vert $, then (41)
holds and hence the proof follows from Theorem 16
\end{proof}

\begin{conclusion}
Theorem 4$(d)$ shows that the main part of the spectrum of $L(q)$ with
PT-symmetric periodic potential is real. Indeed, since $\varepsilon
_{n}\rightarrow0$ and $\delta_{n}\rightarrow0$ as $n\rightarrow\infty$, the
length of the real part $\left\{  \lambda_{n}(t):t\in\lbrack\varepsilon
_{n},\pi-\delta_{n}]\right\}  $ of $\Gamma_{n}$ for $n>N(h)$ is of order
$4\pi^{2}n,$ while the length of the nonreal tails $\left\{  \lambda
_{n}(t):t\in\lbrack0,\varepsilon_{n})\right\}  $ and $\left\{  \lambda
_{n}(t):t\in(\pi-\delta_{n},\pi]\right\}  $ is $O(n\varepsilon_{n})$ and
$O(n\delta_{n})$ respectively. Moreover, instead of (14) using more sharp
estimation one can show that $\varepsilon_{n}$ and $\delta_{n\text{ }}$
approach zero rapidly for smooth potential. For example if $q\in S_{p}$ then
$q_{n}=o(n^{-s})$ and the points $\varepsilon_{n},\pi-\delta_{n}$ corresponds
to zero of $D$ defined in (19). It with (25) implies that $\varepsilon
_{n}=o(n^{-s-1}),\delta_{n}=o(n^{-s-1})$. Therefore the length of the nonreal
tails of $\Gamma_{n}$ is $o(n^{-s}).$ Thus the spectrum of the PT-symmetric
Hill operator is asymptotically real and the spectrum of the self-adjoint Hill
operator is completely real. \ Therefore one can say that they have a closely
spectral characteristic.

However, besides the reality of the spectrum the self-adjoint operators have
the following main characteristics: they have orthogonal projections, spectral
resolution of the identity operator, spectral representation and have no
spectral singularities. The reality of spectrum of PT-symmetric operator does
not guarantee to hold the listed characteristics of the self-adjoint operator
for the PT-symmetric operator. For example, if
\begin{equation}
q(x+1)=q(x),\text{ }q\in L_{1}[0,1],\text{ }q_{-n}=0,\text{ }q_{n}%
\in\mathbb{R}\backslash\left\{  0\right\}  ,\text{ }\forall n=1,2,...
\end{equation}
then $q$ is PT-symmetric potential and the spectrum is completely real:
$\sigma(L(q))=[0,\infty).$ Nevertheless $L(q),$ in general, has no orthogonal
projections and has infinitely many spectral singularities (see [8, 26]). In
this example $\lambda_{n}(0)$ and $\lambda_{n}(\pi)$ are the spectral
singularities if they are double eigenvalues and there exist associated
function corresponding to their eigenfunctions. Moreover, it happens for
almost all potentials (43) (see [28]). If $\lambda_{n}(0)$ and $\lambda
_{n}(\pi)$ are simple and nonreal for $n>N(h)$ then $\Gamma_{n}$ has
complexation points $\lambda_{n}(\varepsilon_{n})$ and $\lambda_{n}(\pi
-\delta_{n})$ which are spectral singularities (see Theorem 11$(a))$. Thus, in
this case $L(q)$ also has infinitely many spectral singularities and hence is
not asymptotically spectral operator.

Theorem 11 shows that if $\lambda_{n}(0)$ and $\lambda_{n}(\pi)$ are real and
simple eigenvalues for large $n$ then the PT-symmetric operator $L(q)$ may
have only finitely many spectral singularities and is asymptotically spectral.
Hence the simplicity plus reality of $\lambda_{n}(0)$ and $\lambda_{n}(\pi)$
for large $n,$ namely the condition (29) on the potential $q,$ provide the
PT-symmetric operator $L(q)$ to be close to the self adjoint Hill operator
$L(q)$ for $q\in S_{p}$.

Finally, note that Theorem 11$(b)$ (the equivalence of $1)$ and $3)$) shows
the asymptotic connection of reality of $\sigma(L(q))$ and spectrality of
$L(q)$ for $q\in S_{p}.$ On the other hand, theorems 12 and 13 for Example 1
show that if $ab>0$ and $a\neq b$ then the spectrum of $H(a,b)$ coincides with
the spectrum of the self-adjoint Mathieu-Hill operator, $L(2c\cos2\pi x),$ all
eigenvalues of $H_{t}(a,b)$ for all $t\in\lbrack0,2\pi)$ are real and simple
and hence $\sigma(H(ab))$ has no spectral singularities and complexation
points, while $H(a,b)$ is not a spectral operator. Thus the relation between
the spectrum and spectrality of the PT-symmetric periodic operators is
complicated and depends on the subclasses of the PT-symmetric potentials.
\end{conclusion}


\begin{thebibliography}{99}                                                                                               %


\bibitem {}Z. Ahmed , Energy band structure due to a complex, periodic, PT
-invariant potential Phys. Lett.A, 286 (2001), 231--235.

\bibitem {}Non-Selfadjoint Operators in Quantum Physics: Mathematical Aspects,
First Edition. Edited by Fabio Bagarello, Jean-Pierre Gazeau, Franciszek Hugon
Szafraniec and Miloslav Znojil, John Wiley \& Sons, Inc. Published, 2015.

\bibitem {}C. M. Bender, G. V. Dunne and P. N. Meisinger, Complex periodic
potentials with real band spectra Phys. Lett. A, 252 (1999), 272--276.

\bibitem {}E Caliceti, S Graffi, Reality and non-reality of the spectrum of PT
-symmetric operators: Operator-theoretic criteria, Pramana Journal of Physics,
73 (2009), 241-249

\bibitem {}J. M. Cervero, PT -symmetry in one-dimensional quantum periodic
potentials Phys. Lett. A, 317 (2003), 26--31.

\bibitem {}J. M. Cervero and A. Rodriguez , The band spectrum of periodic
potentials with PT -symmetry J. Phys. A:Math. Gen. 37 (2004), 10167-10177.

\bibitem {}M. S. P. Eastham, The Spectral Theory of Periodic Differential
Operators, New York: Hafner, 1974.

\bibitem {}M. G. Gasymov, Spectral analysis of a class of second-order
nonself-adjoint differential operators, Fankts. Anal. Prilozhen 14 (1980), 14-19.

\bibitem {}F. Gesztesy and V. Tkachenko, A criterion for Hill's operators to
be spectral operators of scalar type, J. Analyse Math., 107 (2009), 287--353.

\bibitem {}H. F. Jones, The energy spectrum of complex periodic potentials of
Kronig--Penney type Phys. Lett. A, 262 (1999), 242-244.

\bibitem {}K.G. Makris, R. El-Ganainy, D.N. Christodoulides, Z.H. Musslimani,
PT-Symmetric Periodic Optical Potentials, Int J Theor Phys, 50 (2011), 1019--1041.

\bibitem {}V. A. Marchenko, \textquotedblright Sturm-Liouville Operators and
Applications,\textquotedblright\ Birkhauser Verlag, Basel, 1986.

\bibitem {}D. C. McGarvey, Differential operators with periodic coefficients
in $L_{p}(-\infty,\infty)$, J. Math. Anal. Appl. 11 (1965), 564-596.

\bibitem {}D. C. McGarvey, Perturbation results for periodic differential
operators in $L_{p}(-\infty,\infty)$, J. Math. Anal. Appl. 12 (1965), 187-234.

\bibitem {}A. Mostafazadeh, Psevdo-hermitian representation of quantum
mechanics, International Journal of Geometric Methods in Modern Physics, 11 (2010),1191-1306.

\bibitem {}F. S. Rofe-Beketov, The spectrum of nonselfadjoint differential
operators with periodic coefficients, Soviet Math. Dokl. 4 (1963), 1563-1566.

\bibitem {}K. C. Shin, On the shape of spectra for non-self-adjoint periodic
Schr\"{o}dinger operators, Journal of Physics A: Mathematical and General, 37
(2004), 8287-8291.

\bibitem {}A. A. Shkalikov, O. A. Veliev, On the Riesz basis property of the
eigen- and associated functions of periodic and antiperiodic Sturm-Liouville
problems, Math. Notes, 85 (2009), 647-660.

\bibitem {}E. C. Titchmarsh, ''Eigenfunction Expansions Associated with Second
Order Differential Equations,'' Part II, Oxford Univ. Press, London, 1958.

\bibitem {}V. A. Tkachenko, Spectral analysis of nonselfadjoint Schrodinger
operator with a periodic complex potential, Soviet Math. Dokl., 5 (1964), 413-415.

\bibitem {}O. A. Veliev, The spectrum and spectral singularities of
differential operators with complex-valued periodic coefficients. Differential
Cprime Nye Uravneniya, {19} (1983), 1316-1324.

\bibitem {}O. A. Veliev, M. Toppamuk Duman, The spectral expansion for a
nonself-adjoint Hill operators with a locally integrable potential, J. Math.
Anal. Appl. 265 (2002), 76-90.

\bibitem {}O. A. Veliev, Asymptotic analysis of non-self-adjoint Hill's
operators, Central European Journal of Mathematics, 11 (2013), 2234-2256.

\bibitem {}O. A. Veliev, Isospectral Mathieu-Hill operators, Letters in
Mathematical Physics, 103 (2013), 919-925.

\bibitem {}O. A. Veliev, Spectral Analysis of the Non-self-adjoint
Mathieu-Hill Operator, arXiv:1202.4735v4, (2013).

\bibitem {}O. A. Veliev, Spectral Problems of a Class of Non-self-adjoint
One-dimensional Schrodinger Operators, J. Math. Anal. Appl., 422 (2015), 1390--1401.

\bibitem {}O. A. Veliev, On the spectral singularities and spectrality of the
Hill's Operator, Operators and Matrices, 10 (2016), 57-71.

\bibitem {}O. A. Veliev, On a Class of Non-self-adjoint Multidimensional
Periodic Schrodinger Operators, arXiv:1604.02061v2 \ (2016).
\end{thebibliography}
\end{document}